\documentclass[12pt,a4paper,reqno]{amsart}

\usepackage[utf8]{inputenc}

\usepackage{amssymb}
\usepackage{amsfonts}
\usepackage{amsmath}
\usepackage[mathscr]{eucal}
\usepackage{comment}
\usepackage{color}

\newtheorem{thm}{Theorem}[section]
\newtheorem{prop}[thm]{Proposition}
\newtheorem{lem}[thm]{Lemma}

\newtheorem{rem}[thm]{Remark}

\numberwithin{equation}{section}

\allowdisplaybreaks

\def\GL{{\operatorname {GL}}}

\def\SO{{\operatorname{SO}}}
\def\U{{\operatorname {U}}}

\def\Re{{\operatorname {Re}}}

\def\leq{\leqslant}
\def\geq{\geqslant}

\def\le{\leq}
\def\ge{\geq}

\def\1{{\bold 1}}

\renewcommand{\a}{\alpha}
\renewcommand{\b}{\beta}

\newcommand{\e}{\epsilon}

\newcommand{\s}{\sigma}

\newcommand{\Acal}{{\mathcal A}}

\newcommand{\Fcal}{{\mathcal F}}

\newcommand{\Ocal}{{\mathcal O}}

\newcommand{\CC}{\mathbb{C}}

\newcommand{\RR}{\mathbb{R}}

\newcommand{\ZZ}{\mathbb{Z}}

\newcommand{\Sp}{\operatorname{Sp}}

\renewcommand{\Re}{\operatorname{Re}}

\newcommand{\supp}{\operatorname{supp}}

\definecolor{myred}{RGB}{137,37,37}

\title[One-level density weighted by special values of $\zeta$]{One-level density of zeros of the Riemann zeta function weighted by its special values}

\author[Sugiyama]{Shingo Sugiyama}
\address{Faculty of Mathematics and Physics,
Institute of Science and Engineering,
Kanazawa University,
Kakumamachi, Kanazawa, Ishikawa 920-1192, Japan
}
\email{s-sugiyama@se.kanazawa-u.ac.jp}

\author[Suriajaya]{Ade Irma Suriajaya}
\address{Faculty of Mathematics, Kyushu University, 744 Motooka, Nishi-ku, Fukuoka 819-0395, Japan}
\email{adeirmasuriajaya@math.kyushu-u.ac.jp}



\date{\today}

\begin{document}

\begin{abstract}
We consider the one-level density of zeros of the Riemann zeta function weighted by its special values on the line $\Re(s)=\sigma$, where $\sigma$ is a fixed real number such that $1/2\le \sigma\le1$.
We show that the density function for the Riemann zeta function changes if and only if $\sigma=1/2$.
\end{abstract}

\maketitle

\section{Introduction}

\subsection{Weighted density conjecture for a family of $L$-functions}

The relation between number theory and physics is an interesting subject and has been studied in several directions.
For example, the behavior of $L$-functions in number theory can be studied using random matrix theory in physics.
Katz and Sarnak \cite{KatzSarnak1,KatzSarnak2}
predicted the density conjecture that
low-lying zeros of $L$-functions in a family should distribute like eigenvalues of random matrices in a compact matrix group.
For a family $\Fcal$ of $L$-functions,
they labeled $\Fcal$ by the symmetry type $G \in \{\U, \Sp, {\rm O}, \SO({\rm even}), \SO({\rm odd})\}$ derived from compact matrix groups.
Here $\U$ denotes {\it unitary} type, $\Sp$ denotes {\it symplectic type}, and $\SO({\rm even})$, $\SO({\rm odd})$ and ${\rm O}$ refer to {\it orthogonal} types.
More concretely, the distribution of low-lying zeros of $L$-functions in $\Fcal$ is expected to be one of
five density functions $W_G$ of eigenvalues of random matrices:
\begin{align}
&W_{\rm U}(x)=1, \qquad
W_{\Sp}(x)=1-\dfrac{\sin 2\pi x}{2\pi x},
\qquad W_{\SO({\rm even})}(x) = 1+\dfrac{\sin 2\pi x}{2\pi x},
\label{density RMT}\\
&
W_{\SO({\rm odd})}(x) = 1-\dfrac{\sin 2\pi x}{2\pi x}+\delta_0,
\qquad W_{\rm O}(x)=1+\dfrac{1}{2}\delta_0, \notag
\end{align}
where $\delta_0$ is the Dirac delta distribution on $\RR$.

Later, inspired by \cite{KowalskiSahaTsimerman} and \cite{KnightlyReno}, the first author \cite{Sugiyama} predicted the {\it weighted density conjecture} on the weighted one-level density and the change of the density function.
As a special form of the weighted density conjecture,
we consider a family $\{L(s, \Acal)\}_{\Acal\in\Fcal}$ of $L$-functions weighted by special values $L(\s, \Acal)$
or $|L(\s, \Acal)|^2$ at $\s \in [1/2,1)$.
Here we assume that our $L$-function $L(s, \Acal)$ has an Euler product expression
over all prime numbers, admits a meromorphic continuation to the whole complex plane, satisfies a functional equation which relates its value at $s$ to that at $1-s$, and in some cases may require several other analytic properties.

The one-level density of low-lying zeros of an $L$-function $L(s, \Acal)$ is defined as
\begin{align}\label{def one-level}D(\Acal, \phi):=\sum_{\rho=1/2+i\gamma}\phi\left(\frac{\log Q_\Acal}{2\pi}\gamma\right),
\end{align}
where $\rho$ runs over the zeros of $L(s,\Acal)$ in the critical strip $0<\Re(s)<1$ counted with multiplicity,
$\phi$ is a Schwartz function on $\RR$ such that the Fourier transform
$$\widehat{\phi}(\xi):=\int_{-\infty}^{\infty}\phi(x)e^{-2\pi i \xi x}dx$$
of $\phi$ is compactly supported, and $Q_\Acal$ is the analytic conductor of $L(s,\Acal)$ defined by the product of data appearing in the functional equation of $L(s,\Acal)$.
We do not impose the general Riemann hypothesis for $L(s,\Acal)$, as $\phi$ is naturally extended to a Paley-Wiener function on $\CC$ by the Fourier inversion, and hence $\phi(\frac{\log Q_\Acal}{2\pi}\gamma)$ is also defined for complex $\gamma$.

A special case of the weighted density conjecture is as follows.
Let $\{L(s,\Acal)\}_{\Acal\in \Fcal}$ be a family of $L$-functions such that $\Fcal$ is a multiset divided in terms of the analytic conductors.
Namely, $\Fcal$ is decomposed as $\Fcal=\bigcup_{k=1}^{\infty}\Fcal_k$ with $\Fcal_k=\{\Acal \in \Fcal \mid Q_\Acal \asymp k\}<\infty$.
Assume that the symmetry type of $\Fcal$ is given by $G\in \{\U, \Sp, \SO({\rm even})\}$ as Katz-Sarnak's density conjecture. Then the following limit formula holds:
\begin{align*}\lim_{k\rightarrow \infty}\frac{1}{\#\Fcal_k}\sum_{\Acal\in \Fcal_k}D(\Acal, \phi)=\int_{-\infty}^{\infty}\phi(x)W_{G}(x)dx,
\end{align*}
where $W_{G}$ is the density function of eigenvalues of random matrices as in \eqref{density RMT}.
Then, the first author \cite{Sugiyama} predicted that,
for any fixed $\s\in [1/2, 1)$, there should exist
a density function $W_{G,\s}$ such that
\begin{align}
&\lim_{k\rightarrow \infty}\frac{1}{\sum_{\Acal \in \Fcal_k} L(\s, \Acal)}\sum_{\Acal \in \Fcal_k} L(\s, \Acal) D(\Acal, \phi) = \int_{-\infty}^{\infty} \phi(x)W_{G, \s}(x)dx, \ \ G=\Sp \text{ or }\ \SO({\rm even}), \label{weighted SpSO}\\
&\lim_{k\rightarrow \infty}\frac{1}{\sum_{\Acal \in \Fcal_k} |L(\s, \Acal)|^2}\sum_{\Acal \in \Fcal_k} |L(\s, \Acal)|^2 D(\Acal, \phi) = \int_{-\infty}^{\infty} \phi(x)W_{{\rm U}, \s}(x)dx, \quad G={\rm U} \label{weighted U}
\end{align}
for any Paley-Wiener functions $\phi$ on $\CC$. Moreover, he predicted that
$W_{G, \s} \neq W_{G}$ should hold
if and only if $\s = 1/2$.

The weighted density conjecture along the horizontal line $1/2\le\s<1$ as described above
has been supported by several pieces of evidence.
The first author \cite{Sugiyama} gave the first result of this type by
considering a symplectic case of \eqref{weighted SpSO}.
He observed the change of the symmetry type for
the family of symmetric square $L$-functions attached to Hilbert modular forms weighted by its special values at $\s \in [1/2,1]$, where the density function is given by
\[
W_{{\rm Sp},\s}(x)
:=
\begin{cases}
W_{\rm Sp}(x), & \s\neq 1/2,\\[1mm]
1+\dfrac{\sin 2\pi x}{2\pi x}-\dfrac{2\sin^2(\pi x)}{(\pi x)^2}, & \s=1/2.
\end{cases}
\]
His observation leads to the weighted density conjecture for a general family of $L$-functions in \cite{Sugiyama}.
Following up on this work, the authors \cite{SugiyamaChacha} considered a unitary case \eqref{weighted U} where $\Fcal$ consists of Dirichlet $L$-functions. They observed the change of the symmetry type as
\begin{align}\label{WUs}
W_{{\rm U},\s}(x) :=
\begin{cases}
W_{\rm U}(x), & \s\neq 1/2, \\[1mm]
1-\dfrac{\sin^2(\pi x)}{(\pi x)^2}, & \s=1/2.
\end{cases}
\end{align}
Although there is no literature on any orthogonal cases of \eqref{weighted SpSO},
Knightly and Reno \cite{KnightlyReno}
studied the weighted one-level density for the standard $L$-functions attached to elliptic modular forms only when $\s=1/2$.
Since the case $\s \in (1/2,1)$ is computed similarly to the case $\s=1/2$,
their result is extended to the general case $\s \in [1/2,1)$ with minor modification. As a consequence, for standard $L$-functions attached to modular forms, we have
$$
W_{{\rm SO(even)},\s}(x)
:=
\begin{cases}
W_{\rm SO(even)}(x), & \s \neq 1/2,\\
W_{\rm Sp}(x), & \s=1/2.
\end{cases}
$$
Furthermore, Fazzari \cite{Fazzari} formulated a weighted one-level density conjecture
for several families of $L$-functions, and gave expected explicit formulas of $W_{G,1/2}$ which are consistent to random matrix theory.
Note that the evidence above coincides with Fazzari's conjectural forms of $W_{G,1/2}$.

\subsection{Main results}

In this article, we consider a continuous family $\{\zeta(s+it)\}_{t>0}$ consisting of
shifted Riemann zeta functions.
We remark that $\zeta_t(s):=\zeta(s+it)$
has an Euler product expression,
admits a meromorphic continuation to the whole complex plane, and
satisfies the functional equation
\begin{align*}
\pi^{-(s+it)/2}\Gamma\left(\frac{s+it}{2}\right)\zeta_t(s) = \pi^{-(1-s-it)/2}\Gamma\left(\frac{1-s-it}{2}\right)\zeta_{-t}(1-s).
\end{align*}
We also remark that $(s-1+it)\zeta_t(s)$ is an entire function of order one, and we may take $|t|$ as the analytic conductor of $\zeta_t$.

In this paper, we consider the low-lying zeros for the family
$$ \{\zeta_t(s)\}_{t>0}=\bigcup_{T>0}\{\zeta_t(s)\}_{t\in [T,2T]}. $$
Fix a large $T>0$ and $t\in \RR$ with $t\in [T,2T]$. As in \eqref{def one-level}, the one-level density of low-lying zeros of $\zeta_t(s)$ is defined by
$$D_T(\zeta_t, \phi)=\sum_{\rho'=1/2+i\gamma'}\phi\left(\frac{\log T}{2\pi}\gamma'\right)=\sum_{\rho=1/2+i\gamma}\phi\left(\frac{\log T}{2\pi}(\gamma-t)\right),$$
where $\rho'=1/2+i\gamma'$ and $\rho=1/2+i\gamma$ each runs over the zeros of $\zeta_t(s)$ and $\zeta(s)$, respectively, in the region $0<\Re(s)<1$, both counted with multiplicity.
We consider the following continuous average of the one-level densities for $\{\zeta_t\}_{t \in [T, 2T]}$:
\begin{align*}
\frac{1}{\int_{T}^{2T}|\zeta_t(\s)|^{2} dt}\int_{T}^{2T}|\zeta_t(\s)|^{2}D_T(\zeta_t,\phi)dt, \qquad 1/2\le \s\le 1.
\end{align*}
This is an analogue of the discrete average presented on the left-hand side of \eqref{weighted U}.
We recall the unweighted one-level density for $\{\zeta_t\}_{t>0}$ given
by Conrey \cite{Conrey}.

\begin{thm}[{\cite[\S5.3.1]{Conrey}}]\label{Conrey result}
For any Schwartz function $\phi$ on $\RR$ with $\supp(\widehat{\phi})\subset[-2,2]$,
we have
$$\frac{1}{T}\int_{T}^{2T}D(\zeta_t,\phi) dt =\int_{-\infty}^{\infty} \phi(x)W_{\rm U}(x)dx+\Ocal\left(\frac{1}{\log T}\right)$$
as $T\rightarrow \infty$.
Hence the symmetry type of $\{\zeta_t\}_{t>0}$ is unitary.
\end{thm}
We remark that Conrey \cite[\S5.3.1]{Conrey} also mentioned that the support condition $\supp(\widehat{\phi})\subset[-2,2]$
in Theorem \ref{Conrey result}
can be relaxed to $\supp(\widehat{\phi})\subset [-a,a]$ for an arbitrary $a>0$ if we assume the Riemann hypothesis.
We also remark that a smoothed average version of Theorem \ref{Conrey result} is known by the work of Hughes and Rudnick \cite{HughesRudnick}.

Our result on the weighted one-level density for $\{\zeta_t\}_{t >0}$ is stated as follows.
\begin{thm}\label{main result}
Let $\s$ be a real number satisfying $1/2\le \s\le 1$,
and set
$$\a_\s:=\frac{\s}{2\s+1}.$$
Then for any Schwartz function $\phi$ on $\RR$ with $\supp(\widehat\phi)\subset (-\a_\s,\a_\s)$,
we have
$$\frac{1}{\int_{T}^{2T}|\zeta(\s+it)|^2dt}\int_{T}^{2T} |\zeta(\s+it)|^{2}D_T(\zeta_t,\phi)dt=
\int_{-\infty}^{\infty}\phi(x)W_{{\rm U},\s}(x)dx+\Ocal\left(\frac{1}{\log T}\right)
$$
as $T\rightarrow \infty$,
where $W_{{\rm U},\s}$ is the same function as \eqref{WUs}.
\end{thm}
By Theorem \ref{main result}, we conclude that the symmetry type of the one-level density for $\{\zeta_t\}_{t >0}$
weighted by $|\zeta_t(\s)|^2$ changes if and only if $\s=1/2$.
This is evidence toward
the weighted density conjecture \eqref{weighted U} for unitary families.

We introduce some known results for the weighted one-level density for $\{\zeta_t\}_{t>0}$.
Fazzari \cite[Theorem 2.1]{Fazzari} proved the same result as Theorem \ref{main result} only when $\s=1/2$, with a relaxed support condition on $\widehat\phi$ but assuming the ratios conjecture and the Riemann hypothesis\footnote{The ratios conjecture for the Riemann zeta function implies the Riemann hypothesis by \cite[Theorem 5.1]{ConreySnaith} and \cite{BettinGonek}.}.
Later, Bettin and Fazzari \cite[Theorem 1.2]{BettinFazzari} gave smoothed average versions of the density of zeros
weighted by $|\zeta(1/2+it)|^{2k}$ with $k=1,2$
without the Riemann hypothesis.
Contrary to their work, we treat the natural average in $t$ and 
special values of $\zeta_t$ at any $\s\in [1/2,1)$.

{\bf Notation.}
In this article, the symbol $\ZZ_{>0}$ stands for the set of all positive integers.
The letter $\gamma_{\rm E}$ denotes the Euler-Mascheroni constant.
We set $\psi(s):=\Gamma'(s)/\Gamma(s)$, which is known as the digamma function.

\section{Twisted second moments of the Riemann zeta function}

In this section, we give asymptotic formulas for the twisted second moment of the Riemann zeta function
\[ I_m(\s,T):=\int_{T}^{2T}|\zeta(\s+it)|^{2}m^{it}dt,\qquad m \in \ZZ_{>0} \]
for any $\s\in[1/2,1)$.
For the untwisted second moment $I(\s,T):=I_1(\s,T)$, the following asymptotic formula is known:
\begin{equation*}
\int_T^{2T}|\zeta(\s+it)|^2dt = \zeta(2\s)T+\Ocal(T^{2-2\s}\log T), \qquad T\rightarrow \infty,
\end{equation*}
for any $\s\in(1/2,1)$.
This is a classical result which dates back to Landau \cite{lan09}
and Schnee \cite[Section 7]{sch09} in 1909, see also \cite[Equation (8.112)]{IvicBook}.

On the critical line, we have the following estimate
\[ \int_{0}^{T}|\zeta(1/2+it)|^2dt = T\log T + (2\gamma_{\rm E}-1-\log(2\pi))T + \Ocal(T^{1/2}\log{T}), \qquad T \rightarrow \infty, \]
see \cite[Theorem A']{IngMV},
which immediately implies
\begin{equation*}
\int_{T}^{2T}|\zeta(1/2+it)|^2dt = T\log T + (2\gamma_{\rm E} - 1 + \log2 - \log\pi)T + \Ocal(T^{1/2}\log{T}), \quad T\rightarrow \infty.
\end{equation*}
The asymptotic estimate was first obtained by Hardy and Littlewood in 1916 \cite[Theorem 2.41]{HL16} and reproved in their 1923 paper \cite[Theorem E]{HL23} using a refined method they developed in their 1921 paper \cite{HL21}, which allows one to estimate the size of the error.
As mentioned in \cite[Footnote $\ddagger$ in p. 274]{IngMV} and also cited in \cite{HL29}, 
the formula above is originally due to Littlewood (1922), with a weaker error estimate.

The second moment estimate we need is the following twisted version for the Riemann zeta function at $\s\in (1/2,1]$ due to Selberg \cite{Selberg}.
\begin{lem}[{\cite[Lemma 6]{Selberg}}]
\label{Selberg original}
Let $T>10$ and $\sqrt{T}\le H \le T$. For $m,n \in \ZZ_{>0}$ such that $m<T/10$, $n<T/10$ and $\gcd(m,n)=1$,
we have for $1/2<\s \le 1$,
\begin{equation} \label{twisted moment at s}
\begin{aligned}
&\int_{T}^{T+H}|\zeta(\s+it)|^2 \left(\frac{m}{n}\right)^{it}dt \\
&\qquad=\int_T^{T+H}\left\{(mn)^{-\s}\zeta(2\s)+(mn)^{\s-1}
\left(\frac{t}{2\pi}\right)^{1-2\s}\zeta(2-2\s)\right\}dt \\
&\qquad\qquad+\Ocal(T^{1-\s}mn(\log T)^2)+\Ocal\left(HT^{-\s/2}\left\{\left(\frac{m}{n}\right)^{\s}+\left(\frac{m}{n}\right)^{-\s}\right\}(\log T)^2\right).
\end{aligned}
\end{equation}
\end{lem}

Applying Lemma \ref{Selberg original}, we immediately obtain the following asymptotic estimate for $I_m(\s,T)$.
\begin{lem}\label{Selberg sigma}
Let $T>10$ and $m \in \ZZ_{>0}$. Then for $m<T/10$ and $1/2<\s \le 1$,
we have
\begin{align}
I_m(\s,T) 
&= m^{-\s}\zeta(2\s)T+m^{\s-1}
\frac{(2\pi)^{2\s-1}\zeta(2-2\s)(2^{2-2\s}-1)}{2-2\s}T^{2-2\s} \label{main Im(s)}\\
&\qquad+\Ocal(mT^{1-\s}(\log T)^2)+\Ocal\left(m^{\s}T^{1-\s/2}(\log T)^2\right).\label{error Im(s)}
\end{align}
\end{lem}

Next, we extend the result of Lemma \ref{Selberg original} to $\s=1/2$.
\begin{lem}\label{Selberg 1/2}
Let $T>10$ and $\sqrt{T}\le H \le T$.
Then for $m,n \in \ZZ_{>0}$ satisfying $m<T/10$, $n<T/10$ and $\gcd(m,n)=1$, we have
\begin{equation}\label{eq:Selberg_1/2}
\begin{aligned}
&\int_{T}^{T+H}|\zeta(1/2+it)|^2 \left(\frac{m}{n}\right)^{it}dt\\
&\qquad= (2\gamma_{\rm E} -\log 2\pi -\log mn)
(mn)^{-1/2}H
\\
&\qquad\qquad
+(mn)^{-1/2}\{(T+H)\log (T+H)-T\log T-H\}
\\
&\qquad\qquad
+\Ocal(T^{1/2}mn(\log T)^2)+\Ocal\left(HT^{-1/4}
\left\{\left(\frac{m}{n}\right)^{1/2}+\left(\frac{m}{n}\right)^{-1/2}\right\}(\log T)^2\right).
\end{aligned}
\end{equation}
In particular, taking $H=T$ and $n=1$, we have
\begin{align}
I_m(1/2, T)
&= m^{-1/2}T\log T
-m^{-1/2}(\log m )T + (2\gamma_{\rm E} + \log 2 - \log \pi - 1) m^{-1/2}T
\label{main Im(1/2)}\\
&\qquad+\Ocal(mT^{1/2}(\log T)^2)+\Ocal(m^{1/2}T^{3/4}(\log T)^2).
\label{error Im(1/2)}
\end{align}
\end{lem}
\begin{proof}
Taking the limit of \eqref{twisted moment at s} as $\s\rightarrow 1/2$, we have
\begin{align*}
&\int_{T}^{T+H}|\zeta(1/2+it)|^2 \left(\frac{m}{n}\right)^{it}dt\\
&\qquad= \int_T^{T+H}\lim_{\s \rightarrow 1/2}\left\{(mn)^{-\s}\zeta(2\s)+(mn)^{\s-1}\left(\frac{t}{2\pi}\right)^{1-2\s}\zeta(2-2\s)\right\}dt\\
&\qquad\qquad+\Ocal(T^{1/2}mn(\log T)^2)+\Ocal\left(HT^{-1/4}\left\{\left(\frac{m}{n}\right)^{1/2}+\left(\frac{m}{n}\right)^{-1/2}\right\}(\log T)^2\right).
\end{align*}
Recall
\begin{align}\label{zeta asymp}\zeta(s)=\frac{1}{s-1}+\gamma_{\rm E}+\Ocal(s-1), \quad s \rightarrow 1,\end{align}
which implies
\begin{align*}
(mn)^{-\s}\zeta(2\s) = \frac{(mn)^{-\s}}{2\s-1}+\gamma_{\rm E}(mn)^{-\s}+\Ocal(mn^{-\s}(\s-1/2)), \qquad \s\rightarrow 1/2,
\end{align*}
Similarly, we have by taking $s=2-2\s$ in \eqref{zeta asymp},
\begin{align*}
(mn)^{\s-1}\left(\frac{t}{2\pi}\right)^{1-2\s}\zeta(2-2\s)
&= \frac{(mn)^{\s-1}(\frac{t}{2\pi})^{1-2\s}}{1-2\s}
+\gamma_{\rm E}(mn)^{\s-1}\left(\frac{t}{2\pi}\right)^{1-2\s}
\\
&\qquad+\Ocal\left((mn)^{\s-1}\left(\frac{t}{2\pi}\right)^{1-2\s}(\s-1/2)\right), \quad \s \rightarrow 1/2.
\end{align*}
Hence, combining the above we see that as $\s\rightarrow 1/2$,
\begin{align*}
&
(mn)^{-\s}\zeta(2\s)+(mn)^{\s-1}\left(\frac{t}{2\pi}\right)^{1-2\s}\zeta(2-2\s)
\\
&\qquad=
(mn)^{-\s}\frac{(\frac{t}{2\pi mn})^{1-2\s}-1}{1-2\s}+\gamma_{\rm E}(mn)^{-\s}+\gamma_{\rm E}(mn)^{\s-1}\left(\frac{t}{2\pi}\right)^{1-2\s}+\Ocal_{m,n,t}(\s-1/2).
\end{align*}
Taking the limit as $\s\rightarrow 1/2$, this converges to
\[ (mn)^{-1/2}\log\left(\frac{t}{2\pi m n}\right)+2\gamma_{\rm E}(mn)^{-1/2}. \]
Finally, applying $\int \log t \,dt = t\log t - t$, we immediately obtain \eqref{eq:Selberg_1/2},
and we are done.
\end{proof}

\begin{thm}\label{Im(s)/I(s)}
Let $1/2<\s\le 1$ and
$m\in \ZZ_{>0}$.
Then for $T>10$ and $m<T/10$, we have 
\begin{align*}
\frac{I_m(\s, T)}{I(\s,T)}=m^{-\s}
+\Ocal_\s(m^{\s-1}T^{1-2\s})+\Ocal_\s(mT^{-\s}(\log T)^2)+\Ocal_\s(m^{\s}T^{-\s/2}(\log T)^2).
\end{align*}
\end{thm}
\begin{proof}
We write \eqref{main Im(s)} and \eqref{error Im(s)} in Lemma \ref{Selberg sigma} as $F_m$ and $E_m$, respectively.
Then the ratio we want to evaluate is
\begin{align}\label{decomp of Im(s)}
\frac{I_m(\s, T)}{I(\s,T)}=\frac{F_m+E_m}{F_1+E_1}=\frac{F_m}{F_1}+\frac{E_m-\frac{F_m}{F_1}E_1}{F_1+E_1},
\end{align}
We remark that the decomposition \eqref{decomp of Im(s)} is also used in
\cite[Proposition 3.1]{KnightlyReno}, \cite[Corollary 2.10]{Sugiyama} and \cite[Equation (2.3)]{SugiyamaChacha}.
We notice that for $\s\in(1/2,1)$,
\[
\frac{(2\pi)^{2\s-1}(2^{2-2\s}-1)}{2-2\s}
=
(2\pi)^{2\s-1}(\log2+\Ocal(\s-1))
\asymp 1,
\]
thus this is bounded on $(1/2,1]$.
Since
$\zeta(2-2\s)/\zeta(2\s)$ is also bounded on $(1/2,1]$,
the first term of the right-hand side of \eqref{decomp of Im(s)} is
evaluated as
\begin{align*}\frac{F_m}{F_1}
&= \frac{m^{-\s}\zeta(2\s)T+\Ocal(\zeta(2-2\s)m^{\s-1}T^{2-2\s})}{\zeta(2\s)T+\Ocal\left(\zeta(2-2\s) T^{2-2\s}\right)}
= \frac{m^{-\s} + \Ocal(m^{\s-1}T^{1-2\s})}{1 + \Ocal\left(T^{1-2\s}\right)}
\\
&= \left\{m^{-\s}+\Ocal(m^{\s-1}T^{1-2\s})\right\} \times \left\{1+\Ocal(T^{1-2\s})\right\}\\
&= m^{-\s}+\Ocal(m^{\s-1}T^{1-2\s}).
\end{align*}
The second term of the right-hand side of \eqref{decomp of Im(s)} is
estimated as
\begin{align*} &\frac{E_m-\frac{F_m}{F_1}E_1}{F_1+E_1}\\
&\qquad\ll_\s \frac{mT^{1-\s}(\log T)^2 + m^{\s}T^{1-\s/2}(\log T)^2
+(m^{-\s}+m^{\s-1}T^{1-2\s}) T^{1-\s/2}(\log T)^2}{T}\\
&\qquad\ll
mT^{-\s}(\log T)^2+m^{\s}T^{-\s/2}(\log T)^2
+m^{\s-1}T^{1-5\s/2}(\log T)^2
\\
&\qquad\ll
mT^{-\s}(\log T)^2 + m^{\s}T^{-\s/2}(\log T)^2,
\end{align*}
since $1-2\s<0$ for $\s\in(1/2,1]$.
This completes the proof.
\end{proof}

On the critical line, we have the following.
\begin{thm}\label{Im(1/2)/I(1/2)}
Let $T>10$. For $m \in \ZZ_{>0}$ satisfying $m<T/10$, we have
\begin{align*}
\frac{I_m(1/2, T)}{I(1/2,T)} &= m^{-1/2} - m^{-1/2}\frac{\log m}{\log T}
+\Ocal\left(\frac{m^{-1/2}\log m}{(\log T)^2}\right)\\
&\qquad+\Ocal(mT^{-1/2}\log T)+\Ocal(m^{1/2}T^{-1/4}\log T).
\end{align*}
\end{thm}
\begin{proof}
We write \eqref{main Im(1/2)} and \eqref{error Im(1/2)} in Lemma \ref{Selberg 1/2} as $F_m$ and $E_m$, respectively.
Then the equation \eqref{decomp of Im(s)} is also valid when $\s=1/2$.
Setting $C:=2\gamma_{\rm E} -\log \pi -1+\log 2$,
we can simplify the first term on the most right-hand side of \eqref{decomp of Im(s)} for $\s=1/2$ to
\begin{align*}
\frac{F_m}{F_1}
&= \frac{m^{-1/2}T\log T - m^{-1/2}(\log m)T + Cm^{-1/2}T}{T\log T+CT}
\\
&= m^{-1/2}\frac{\log T - \log m + C}{\log{T}+C}
= m^{-1/2}\left( 1 - \frac{\log m}{\log{T}+C} \right)
\\
&= m^{-1/2}\left( 1 - \frac{\log m}{\log{T}}
\left\{
1 + \Ocal\left(\frac{1}{\log{T}}\right)
\right\}
\right)
\\
&=m^{-1/2}-m^{-1/2}\frac{\log m}{\log T}
+\Ocal\left(\frac{m^{-1/2}\log m}{(\log T)^2}\right).
\end{align*}
Finally, we evaluate the second term on the most right-hand side of \eqref{decomp of Im(s)} for $\s=1/2$ as
\begin{align*}\frac{E_m-\frac{F_m}{F_1}E_1}{F_1+E_1}
&\ll \frac{mT^{1/2}(\log T)^2 + m^{1/2}T^{3/4}(\log T)^2+m^{-1/2}(1+\frac{\log m}{\log T}) T^{3/4}(\log T)^2}{T\log T}\\
&\ll mT^{-1/2}\log T+m^{1/2}T^{-1/4}\log T.
\end{align*}
This completes the proof.
\end{proof}

\begin{rem}
The main terms in Theorems \ref{Im(s)/I(s)} and
\ref{Im(1/2)/I(1/2)} have the same forms
as those in \cite[Propositions 2.2 and 2.3]{SugiyamaChacha}, respectively.
\end{rem}

\section{Weighted one-level density for the Riemann zeta function}

In this section, we compute the one-level density of
low-lying zeros of $\{\zeta_t\}_{t >0}$
weighted by $\zeta_t(\s)$ for any fixed $\s \in [1/2,1]$.
In what follows, we suppress the dependence on our test functions $\phi$ in the $\Ocal$-symbols.

Set
\[ \Omega(r) := \frac{1}{2}\psi\left(\frac{1}{4}+\frac{ir}
{2}\right)+\frac{1}{2}\psi\left(\frac{1}{4}-\frac{ir}{2}\right)-\log \pi, \qquad r \in \RR. \]
Let $\phi$ be a Schwartz function on $\RR$ such that $\supp \widehat{\phi}$ is compact.
With the aid of \cite[Lemma 2.1]{HughesRudnick},
we have the following explicit formula for the Riemann zeta function:
\begin{align*}D_{T}(\zeta_t, \phi) 
&= \frac{1}{2\pi}\int_{-\infty}^{\infty}\phi\left(\frac{\log T}{2\pi}(r-t)\right)
\Omega(r)dr\\
&\qquad+ \phi\bigg(\frac{\log T}{2\pi}\bigg(-\frac{i}{2}-t\bigg)\bigg)
+\phi\bigg(\frac{\log T}{2\pi}\bigg(\frac{i}{2}-t\bigg)\bigg)\\
&\qquad- \frac{1}{\log T}\sum_{p}\sum_{m=1}^{\infty}\frac{\log p}{p^{m/2}}\left(\widehat{\phi}\left(\frac{m\log p}{\log T}\right)p^{-imt}+
\widehat{\phi}\left(-\frac{m\log p}{\log T}\right)p^{imt}\right),
\end{align*}
where $p$ runs over the set of all prime numbers.
From this, we obtain the asymptotics
\begin{equation}\label{eq:DTzeta}
\begin{aligned}
D_{T}(\zeta_t, \phi)
&= \frac{1}{2\pi}\int_{-\infty}^{\infty}\phi\left(\frac{\log T}{2\pi}(r-t)\right)
\Omega(r)dr
+\Ocal\left(\frac{1}{\log T}\right) \\
&\qquad-\frac{1}{\log T}\sum_{p}\frac{\log p}{p^{1/2}}\left(\widehat{\phi}\left(\frac{\log p}{\log T}\right)p^{-it}+
\widehat{\phi}\left(-\frac{\log p}{\log T}\right)p^{it}\right)\\ 
&\qquad-\frac{1}{\log T}\sum_{p}\frac{\log p}{p}\left(\widehat{\phi}\left(\frac{2\log p}{\log T}\right)p^{-2it}+\widehat{\phi}\left(-\frac{2\log p}{\log T}\right)p^{2it}\right),
\end{aligned}
\end{equation}
as $T\rightarrow\infty$.

By Stirling's formula, we have 
$$
\Omega(t) = \log\frac{|t|}{2\pi} + \Ocal\left(\frac{1}{1+|t|^2}\right), \qquad |t|\ge 1
$$
and for $|x|<(2\pi)^{-1}|t|\log{T}$,
$$
\Omega\left(t+\frac{2\pi}{\log T}x\right)
=
\Omega(t)+\Ocal\left(\frac{1}{1+|t|}\frac{|x|}{\log T}\right), \qquad |t|\ge 1,
$$
see also \cite[Proof of Lemma 2.2]{HughesRudnick}.
Hence
$$
\frac{1}{2\pi}\int_{-\infty}^{\infty}\phi\left(\frac{\log T}{2\pi}(r-t)\right)\Omega(r)dr
=
\frac{\log{|t|}-\log{2\pi}}{\log T}\int_{-\infty}^{\infty}\phi(x)dx+\Ocal\left(\frac{1}{1+|t|}\frac{1}{\log T}\right).
$$
Substituting this into \eqref{eq:DTzeta}, we obtain, as $T\rightarrow\infty$,
\begin{align}
&\frac{1}{I(\s,T)}\int_{T}^{2T}
|\zeta(\s+it)|^{2}D_{T}(\zeta_t, \phi)dt \notag\\
&\qquad= \frac{\widehat{\phi}(0)}{I(\s,T)\log T}\int_{T}^{2T}|\zeta(\s+it)|^2\log{t}\,dt
\notag \\
&\qquad\qquad-\frac{1}{\log T}\sum_{p}\frac{\log p}{p^{1/2}}\bigg\{\widehat{\phi}\left(\frac{\log p}{\log T}\right)
\frac{\overline{I_p(\s,T)}}{I(\s,T)}+
\widehat{\phi}\left(-\frac{\log p}{\log T}\right)\frac{I_{p}(\s,T)}{I(\s,T)}\bigg\}
\notag \\
&\qquad\qquad-\frac{1}{\log T}\sum_{p}\frac{\log p}{p}
\bigg\{\widehat{\phi}\left(\frac{2\log p}{\log T}\right)
\frac{\overline{I_{p^2}(\s,T)}}{I(\s,T)}+
\widehat{\phi}\left(-\frac{2\log p}{\log T}\right)\frac{I_{p^2}(\s,T)}{I(\s,T)}\bigg\}
\notag \\
&\qquad\qquad+\Ocal\left(\frac{1}{\log T}\right)
\notag \\
&\qquad=: M_0(\s,T) + M_1(\s,T) + M_2(\s,T) + \Ocal\left(\frac{1}{\log T}\right).
\label{eq:1leveldensity}
\end{align}

We first deal with the easiest term, $M_0(\s,T)$, for all $1/2\le\s\le1$.
\begin{lem}
\label{formula of M0}
Let $1/2\le\s\le1$.
Then for any Schwartz function $\phi$ on $\RR$, we have
$$M_0(\s,T)=\widehat{\phi}(0)+\Ocal\left(\frac{1}{\log{T}}\right), \qquad T\rightarrow \infty.$$
\end{lem}

\begin{proof}
First, we note that $\log t=\log T+\Ocal(1)$ for $t \in [T,2T]$, thus for any $1/2\le\s\le1$, we have
\begin{align*}
\int_{T}^{2T}|\zeta(\s+it)|^2\log t \,dt
&= (\log T) \int_{T}^{2T}|\zeta(\s+it)|^2 dt
+\Ocal\left(\int_{T}^{2T}|\zeta(\s+it)|^2 dt\right)
\\ &=
I(\s,T)
\left(\log T+\Ocal(1)\right).
\end{align*}
Hence, we immediately see that
\[
M_0(\s,T)
=
\widehat{\phi}(0)
\left(1+\Ocal\left(\frac{1}{\log T}\right)\right)
=
\widehat{\phi}(0)
+
\Ocal\left(\frac{1}{\log T}\right),
\]
which concludes the proof.
\end{proof}

Next we prove Theorem \ref{main result} in the case $1/2<\s\le1$.
We write it separately as a proposition.

\begin{prop}\label{formula of weighted density by s}
Let $1/2<\s\le1$ and
set
\[ \a_\s := \dfrac{\s}{2\s+1}. \]
Then for any Schwartz function $\phi$ on $\RR$ such that $\supp \widehat{\phi}\subset (-\a_\s, \a_\s)$, we have
\begin{align*}
\frac{1}{I(\s,T)}\int_{T}^{2T}
|\zeta(\s+it)|^{2}D_{T}(\zeta_t, \phi)dt
=\widehat{\phi}(0)
+\Ocal_{\s}\left(\frac{1}{\log T}\right), \qquad T\rightarrow \infty.
\end{align*}
\end{prop}

\begin{proof}
Recall the equality \eqref{eq:1leveldensity}.
The first term $M_0(\s,T)$ is given by Lemma \ref{formula of M0}.
Hence it suffices to show that, if $\supp\widehat{\phi}\subset [-\a,\a]$ with $\a<\a_\s$, we have
\begin{equation}\label{M1M2est}
    M_1(\s,T)+M_2(\s,T) \ll_\s \frac{1}{\log{T}}.
\end{equation}
In what follows, we let $\supp\widehat{\phi}\subset [-\a,\a]$ for some $\a \in (0,1)$.

We first compute the contribution of $M_1(\s,T)$.
In the summation, $p$ varies over the primes with $p\le T^\a$ by the support condition on $\widehat\phi$.
Then since $\a<1$, all such primes satisfy $p<T/10$ for any sufficiently large $T$.
Thus Theorem \ref{Im(s)/I(s)} yields
\begin{equation}\label{M_1(s) formula}
\begin{aligned}
M_1(\s,T)
&= -\frac{1}{\log T}\sum_{p}\frac{\log p}{p^{1/2}}
\left\{\widehat{\phi}\left(\frac{\log p}{\log T}\right)
+\widehat{\phi}\left(-\frac{\log p}{\log T}\right)\right\}\\
&\qquad\times (p^{-\s}
+\Ocal_\s(p^{\s-1}T^{1-2\s}+pT^{-\s}(\log T)^2+p^{\s}T^{-\s/2}(\log T)^2)).
\end{aligned}
\end{equation}
Since $\s>1/2$, the main term of \eqref{M_1(s) formula} is bounded by
\begin{align*}
-\frac{1}{\log T}\sum_{p}\frac{\log p}{p^{1/2}}\left\{\widehat\phi\left(\frac{\log p}{\log T}\right)+\widehat\phi\left(-\frac{\log p}{\log T}\right)\right\}p^{-\s}
\ll \frac{1}{\log T}\sum_{p}\frac{\log p}{p^{1/2+\s}} \ll_\s \frac{1}{\log T}.
\end{align*}

By the prime number theorem, applying partial summation, we have for any fixed $a>-1$,
\begin{align}\label{PNT}
\sum_{p\le x}p^{a}\log p = \frac{1}{a+1}x^{a+1}+\Ocal\left(\frac{x^{a+1}}{\log x}\right)=\Ocal_a(x^{a+1}), \quad x \rightarrow \infty.
\end{align}
Applying \eqref{PNT} to
the error terms of \eqref{M_1(s) formula}, we immediately get that
each of the bounding terms is further bounded as follows:
$$ \frac{T^{1-2\s}}{\log T}\sum_{p\le T^\a}p^{\s-3/2}\log p\ll \frac{T^{1-2\s+(\s-1/2)\a}}{\log T},
$$
\begin{align}\label{second error M1}
T^{-\s}\log T\sum_{p\le T^\a}p^{1/2}\log p\ll_\e T^{-\s+3\a/2+\e},
\end{align}
\begin{align}\label{third error M1}
T^{-\s/2}\log T\sum_{p\le T^\a}p^{\s-1/2}\log p\ll_\e T^{-\s/2+\a(\s+1/2)+\e}
\end{align}
for any $\e >0$.
Hence, each of the bounds above is bounded by $(\log{T})^{-1}$, up to constant multiples, if
\[
\a < \min\left\{2, \frac{2(\s-\e)}{3}, \frac{\s-2\e}{2\s+1}\right\}.
\]
Note that for any sufficiently small $\e>0$, we have
$$
\frac{2(\s-\e)}{3}-\frac{\s-2\e}{2\s+1}
=
\frac{\s(4\s-1)+4\e(1-\s)}{3(2\s+1)}>0,
$$
and also, since $1/2<\s\le1$,
\[
\frac{\s-2\e}{2\s+1} < \frac{1}{2}.
\]
Hence, if
$$ \a < \frac{\s-2\e}{2\s+1} $$
for any sufficiently small $\e>0$,
then
$$ M_1(\s,T) \ll_{\s,\e} \frac{1}{\log{T}}. $$

Next we estimate $M_2(\s,T)$.
By the support condition on $\widehat{\phi}$,
only prime numbers $p$ satisfying $p\le T^{\a/2}$ ($\a<1$) appear in the summation.
Applying Theorem \ref{Im(s)/I(s)}, we have
\begin{equation}\label{M_2(s) formula}
\begin{aligned}
M_2(\s,T)
&= -\frac{1}{\log T}\sum_{p}\frac{\log p}{p}\left\{\widehat{\phi}\left(\frac{2\log p}{\log T}\right)
+\widehat{\phi}\left(-\frac{2\log p}{\log T}\right)\right\}
\\
&\qquad\times (p^{-2\s} 
+\Ocal_\s(p^{2\s-2}T^{1-2\s}+p^2T^{-\s}(\log T)^2+p^{2\s}T^{-\s/2}(\log T)^2)).
\end{aligned}
\end{equation}
As in the case of $M_1(\s,T)$, since $\s>1/2$, the main term of \eqref{M_2(s) formula} is $\Ocal_\s((\log{T})^{-1})$.
Again, applying \eqref{PNT}, we get that each of the error terms of \eqref{M_2(s) formula}
is bounded, up to constant multiples, by
$$ \frac{T^{1-2\s}}{\log T}\sum_{p\le T^{\a/2}}p^{2\s-3}\log p\ll \frac{T^{1-2\s+(\s-1)\a}}{\log T}<\frac{1}{\log T},
$$
\begin{align}\label{second error M2}
T^{-\s}\log T\sum_{p\le T^{\a/2}}p\log p\ll_\e T^{-\s+\a+\e},
\end{align}
\begin{align}\label{third error M2}
T^{-\s/2}\log T\sum_{p\le T^{\a/2}}p^{2\s-1}\log p\ll_\e T^{-\s/2+\s\a+\e}
\end{align}
for any $\e >0$, where we have used the condition $1/2<\s\le1$ in the first inequality above.
The other two bounds are also bounded by $(\log{T})^{-1}$, up to constant multiples,
as long as $\a<\min\{\s-\e,1/2-\e/\s\}=1/2-\e/\s$, because $1/2<\s\le1$.
Therefore,
\[
M_2(\s,T)\ll_{\s,\e} \frac{1}{\log T}
\]
holds if $\a<1/2-\e/\s$ for any sufficiently small $\e>0$.

Finally, note that for any sufficiently small $\e>0$, we have
$$\frac{1}{2}-\frac{\e}{\s}-\frac{\s-2\e}{2\s+1}
=\frac{\s-2\e}{2\s(2\s+1)}
>0.
$$
Thus \eqref{M1M2est} holds if $\supp\widehat{\phi}\subset [-\a,\a]$ with $\a<\a_\s$.
This completes the proof.
\end{proof}

Next we consider the case $\s=1/2$ of Theorem \ref{main result}.
It follows immediately from the following proposition.

\begin{prop}
For any Schwartz function $\phi$ on $\RR$ such that $\supp \widehat{\phi}\subset (-1/4, 1/4)$, we have as $T\rightarrow \infty$,
\begin{align*}
\frac{1}{I(1/2,T)}\int_{T}^{2T}
|\zeta(1/2+it)|^{2}D_{T}(\zeta_t, \phi)dt
&= \widehat{\phi}(0)-\phi(0)+\int_{-\infty}^{\infty} \widehat{\phi}(x)|x|dx
+\Ocal\left(\frac{1}{\log T}\right)\\
&= \int_{-\infty}^{\infty}\phi(x)\left(1-\frac{\sin^2 (\pi x)}{(\pi x)^2}\right)dx+\Ocal\left(\frac{1}{\log T}\right).
\end{align*}
\end{prop}

\begin{proof}
As in the proof of Proposition \ref{formula of weighted density by s}, 
we recall the equality \eqref{eq:1leveldensity} and note that the evaluation of the term $M_0(1/2,T)$ is done in Lemma \ref{formula of M0}.
We deal with $M_1(1/2,T)$ and $M_2(1/2,T)$.
Suppose $\supp \widehat{\phi} \subset [-\b,\b]$ for some $\b\in(0,1)$.

First we consider $M_1(1/2,T)$. Each $p$ in the summation satisfies $p<T/10$ for any sufficiently large $T$ since $\b<1$.
Thus Theorem \ref{Im(1/2)/I(1/2)} yields
\begin{equation}\label{M_1(1/2) formula}
\begin{aligned}
&M_1(1/2,T)
\\ &\qquad= -\frac{1}{\log T}\sum_{p}\frac{\log p}{p^{1/2}}
\left\{\widehat{\phi}\left(\frac{\log p}{\log T}\right) 
+\widehat{\phi}\left(-\frac{\log p}{\log T}\right)\right\}
\\
&\qquad\qquad\times \left\{p^{-1/2}
-\frac{p^{-1/2}\log p}{\log T} \right.
\\
&\qquad\qquad\qquad+ \left.
\Ocal\left(\frac{p^{-1/2}\log p}{(\log T)^2}\right)
+ \Ocal(pT^{-1/2}\log{T}) + \Ocal(p^{1/2}T^{-1/4}\log T)\right\}.
\end{aligned}
\end{equation}
Since this vanishes if $\phi$ is odd, we may assume that
$\phi$ is an even function.
Using partial summation and the prime number theorem,
we can show that as $T\rightarrow \infty$,
\begin{align}\label{partial sum log1}
\sum_{p}\widehat\phi\left(\frac{\log p}{\log T}\right)\frac{\log p}{p\log T}
=\frac{1}{2}\int_{-\infty}^{\infty}\widehat{\phi}(x)dx+\Ocal\left(\frac{1}{\log T}\right)
\end{align}
and
\begin{align}\label{partial sum log2}
\sum_{p}\widehat{\phi}\left(\frac{\log p}{\log T}\right)\frac{(\log p)^2}{p(\log T)^2}
=\frac{1}{2}\int_{-\infty}^{\infty}\widehat{\phi}(x)|x|dx+\Ocal\left(\frac{1}{\log T}\right)
\end{align}
hold for any even Schwartz function $\phi$
(see \cite[Proof of Proposition 3.2]{Sugiyama} or \cite[Proof of Proposition 3.2]{SugiyamaChacha}\footnote{
In \cite[Proposition 3.2]{SugiyamaChacha}, $\phi$ is implicitly assumed to be even. See Remark \ref{correction of Dirichlet L} for details.}).
With the aid of these formulas, the two main terms of \eqref{M_1(1/2) formula}
sum up to
$$-\int_{-\infty}^{\infty}\widehat{\phi}(x)dx
+\int_{-\infty}^{\infty}\widehat{\phi}(x)|x|dx+\Ocal\left(\frac{1}{\log T}\right)=
-\phi(0)+\int_{-\infty}^{\infty}\widehat{\phi}(x)|x|dx+\Ocal\left(\frac{1}{\log T}\right)$$
as $T\rightarrow \infty$, since $\phi$ is an even function.

The contribution of the first error term of \eqref{M_1(1/2) formula}
is $\Ocal((\log{T})^{-1})$ by \eqref{partial sum log2}.
The contributions of the second and the third error terms
of \eqref{M_1(1/2) formula} are evaluated similarly to \eqref{second error M1} and \eqref{third error M1} by using \eqref{PNT}, and
they are estimated as $\Ocal(T^{-1/2+3\b/2})$ and $\Ocal(T^{-1/4+\b})$, respectively.
Hence all the three error terms of \eqref{M_1(1/2) formula} are absorbed into $\Ocal((\log{T})^{-1})$ as long as $\b<1/4$.

Next we evaluate $M_2(1/2,T)$. As in previous arguments, we begin with the estimate
\begin{equation}\label{M_2(1/2) formula}
\begin{aligned}
&M_2(1/2,T)\\
&\qquad= -\frac{1}{\log T}\sum_{p}\frac{\log p}{p}\left\{\widehat{\phi}\left(\frac{2\log p}{\log T}\right) +
\widehat{\phi}\left(-\frac{2\log p}{\log T}\right)\right\} \\
&\qquad\qquad\times \left\{p^{-1}
-\frac{2p^{-1}\log p}{\log T} \right.
\\ &\qquad\qquad\qquad
+
\left.
\Ocal\left(\frac{p^{-1}\log p}{(\log T)^2}\right)
+\Ocal(p^2T^{-1/2}\log T)+\Ocal(pT^{-1/4}\log T)\right\},
\end{aligned}
\end{equation}
which follows from Theorem \ref{Im(1/2)/I(1/2)}.
The contributions of the two main terms and the first error term
of \eqref{M_2(1/2) formula} are simultaneously bounded, up to constant multiples, by
$$
\frac{1}{\log T}\sum_{p}\frac{(\log p)^2}{p^2}\ll \frac{1}{\log T}.
$$
The contributions of the second and the third error terms
of \eqref{M_2(1/2) formula}
are estimated in the same way as in \eqref{second error M2}
and \eqref{third error M2} using \eqref{PNT},
and they are estimated as
$\Ocal(T^{-1/2+\b})$ and $\Ocal(T^{-1/4+\b/2})$, respectively.
Therefore, as long as $\b<1/2$, these two terms
are absorbed into $\Ocal((\log{T})^{-1})$.

All in all, we obtain the first equality of the assertion
under the condition $\supp\widehat{\phi} \subset (-1/4, 1/4)$. The second equality of the assertion follows from \cite[Proposition 3.3]{SugiyamaChacha}.
\end{proof}

\begin{rem}\label{correction of Dirichlet L}
In \cite[\S3]{SugiyamaChacha}, the explicit formula of the Dirichlet $L$-function $L(s,\chi)$ is valid only when the test function $\phi$ is even, and $\phi$ is implicitly assumed to be an even function.
Moreover,
the asymptotic formulas \eqref{partial sum log1} and \eqref{partial sum log2} used in \cite[Proposition 3.2]{SugiyamaChacha}
hold only when $\phi$ is even.
We may assume that $\phi$ is an even function in \cite[\S3]{SugiyamaChacha}
for the following reason.
Let $q$ be a prime number and let $\Fcal_q$ be the set of all non-principal Dirichlet characters modulo $q$.
If $\phi$ is an odd Schwartz function on $\RR$ such that $\widehat{\phi}$ is compactly supported, then the functional equation of $L(s,\chi)$ for any $\chi \in \Fcal_q$ implies
$D(\chi^{-1},\phi)=-D(\chi,\phi)$. Therefore for any $s \in [1/2, 1)$, we have
$$\sum_{\chi\in\Fcal_q}|L(s,\chi)|^2D(\chi,\phi)=\sum_{\chi\in\Fcal_q}|L(s,\chi^{-1})|^2D(\chi^{-1},\phi)=-\sum_{\chi \in \Fcal_q}|L(s,\chi)|^2D(\chi,\phi),$$
thus we obtain
$$\sum_{\chi\in \Fcal_{q}}|L(s,\chi)|^2D(\chi,\phi)=0=\int_{-\infty}^{\infty}\phi(x)W_{{\rm U},s}(x)dx.$$
\end{rem}

\section*{Acknowledgment}

The authors would like to thank Sandro Bettin for informing them of the paper \cite{ConreySnaith}.
The first author was supported by
JSPS KAKENHI Grant Number JP24K06664.
The second author was supported by the Inamori Research Grant 2024 and
JSPS KAKENHI Grant Number 22K13895.



\begin{thebibliography}{99} 

\bibitem{BettinFazzari}
S. Bettin, A. Fazzari, {\it A weighted one-level density of the non-trivial zeros of the Riemann zeta-function}, Math. Z. {\bf 307} (2024), no. 2, Paper No. 32, 23 pp.

\bibitem{BettinGonek} S. Bettin, S.M. Gonek, {\it The $\theta=\infty$ conjecture implies the Riemann hypothesis}, Mathematika {\bf 63} (2017), no. 1, 29--33.

\bibitem{Conrey} B. Conrey, {\it Families of L-functions and 1-level densities},
In: Recent Perspectives on Random Matrix Theory and Number Theory,
London Math. Soc. Lecture Note Ser. {\bf 332}, Cambridge University Press, Cambridge (2005), 225--249.

\bibitem{ConreySnaith} B. Conrey, N. Snaith, {\it Applications of $L$-functions ratios conjecture}, Proc. Lond. Math. Soc. (3) {\bf 94} (2007), no. 3, 594--646.

\bibitem{Fazzari} A. Fazzari, {\it A weighted one-level density of families of $L$-functions}, Algebra Number Theory {\bf 18} (2024), no.1, 87--132.

\bibitem{HL16} G. H. Hardy, J. E. Littlewood, {\it Contributions to the theory of the Riemann zeta-function and the theory of the distribution of primes}, Acta Math. {\bf 41} (1916), no. 1, 119--196.

\bibitem{HL21} G. H. Hardy, J. E. Littlewood, {\it The zeros of Riemann's zeta-function on the critical line}, Math. Z. {\bf 10} (1921), no. 3-4, 283--317.

\bibitem{HL23} G. H. Hardy, J. E. Littlewood, {\it The Approximate Functional Equation in the Theory of the Zeta-Function, with Applications to the Divisor-Problems of Dirichlet and Piltz}, Proc. Lond. Math. Soc. (2) {\bf 21} (1923), 39--74.

\bibitem{HL29} G. H. Hardy, J. E. Littlewood, {\it The Approximate Functional Equations for $\zeta(s)$ and $\zeta^2(s)$}, Proc. Lond. Math. Soc. (2) {\bf 29} (1929), no. 2, 81--97.

\bibitem{HughesRudnick} C.P. Hughes, Z. Rudnick, {\it Linear statistics for zeros of the Riemann's zeta function}, C. R. Acad. Sci. Paris, Ser. {\bf I 335} (2002), 667--670.

\bibitem{IngMV} A. E. Ingham, {\it Mean-Value Theorems in the Theory of the Riemann Zeta-Function}, Proc. Lond. Math. Soc. (2) {\bf 27} (1927), no. 4, 273--300.

\bibitem{IvicBook} A. Ivić, {\it The Riemann zeta-function}, Theory and applications. Reprint of the 1985 original, Dover Publications, Inc., Mineola, NY, 2003.

\bibitem{KatzSarnak1} N. M. Katz, P. Sarnak, {\it Random matrices, Frobenius eigenvalues, and monodromy}, American Mathematical Society Colloquium Publications {\bf 45}, American Mathematical Society, Providence (1999).

\bibitem{KatzSarnak2} N. M. Katz, P. Sarnak, {\it Zeroes of zeta functions and symmetry}, Bull. Amer. Math. Soc. (N.S.) {\bf 36} (1999), no. 1, 1--26.

\bibitem{KnightlyReno} A. Knightly, C. Reno, {\it Weighted distribution of low-lying zeros of $\GL(2)$ $L$-functions}, Canad. J. Math. {\bf 71} (1) (2019), 153--182.

\bibitem{KowalskiSahaTsimerman} E. Kowalski, A. Saha, J. Tsimerman, {\it Local spectral equidistribution for Siegel modular forms and applications}, Compos. Math. {\bf 148} (2012), 335--384.

\bibitem{lan09} E. Landau, {\it Handbuch der {L}ehre von der {V}erteilung der {P}rimzahlen}, 2 {B}\"ande, 2d ed with an appendix by Paul T. Bateman,
Chelsea Publishing Co., New York, 1953.

\bibitem{sch09} W. Schnee, {\it \"Uber Mittelwertsformeln in der Theorie der Dirichlet'schen Reihen}, habilitation thesis, Kgl. Univ. Breslau, 1909.

\bibitem{Selberg} A. Selberg, {\it Contributions to the theory of the Riemann zeta-function}, Arch. Math. Naturvid. {\bf 48} (1946), no. 5, 89--155.

\bibitem{Sugiyama} S. Sugiyama, {\it Low-lying zeros of symmetric power $L$-functions weighted by symmetric square $L$-values}, J. Math. Soc. Japan {\bf 77} (2025), no. 4, 1205--1231.

\bibitem{SugiyamaChacha} S. Sugiyama, A. I. Suriajaya, {\it Weighted one-level density of low-lying zeros of Dirichlet L-functions}, Res. Number Theory {\bf 8} (2022), no. 3, Paper No. 55, 11 pp.


\end{thebibliography}
\end{document}